\documentclass{amsart}
\usepackage{amssymb,amsmath,amsfonts,amsthm}
\usepackage{graphicx}
\graphicspath{ {./images/} }

\usepackage{cleveref}

\usepackage{psfrag}
\usepackage{mathtools}
\usepackage{color}
\usepackage{todonotes}
\usepackage{enumitem}
\usepackage{mathrsfs}

\usepackage{showlabels}

\theoremstyle{plain}
\newtheorem{main}{Theorem}

\newtheorem{assumption}{Assumption}
\newtheorem{maincor}[main]{Corollary}
\newtheorem{theorem}{Theorem}[section]
\newtheorem{lemma}[theorem]{Lemma}
\newtheorem{proposition}[theorem]{Proposition}

\theoremstyle{remark}
\newtheorem{remark}[theorem]{Remark}
\newtheorem{definition}{Definition}

\newcommand\numberthis{\addtocounter{equation}{1}\tag{\theequation}}

\newcommand{\diam}{\operatorname{diam}}
\newcommand{\ess}{\operatorname{ess}}

           \def\ea{\end{array}}
          \def\ec{\end{center}}
     \def\ed{\end{description}}
        \def\ee{\end{equation}}
       \def\eea{\end{eqnarray}}
     \def\eeaa{\end{eqnarray*}}
 \def\et{\end{thebibliography}}

\def\ra{\rightarrow}

\def\cA{{\mathcal A}}

\def\cC{{\mathcal C}}
\def\cO{{\mathcal O}}

\def\cB{{\mathcal B}}

\def\ln{\operatorname{ln}}
\def\TT{{\mathbb T}}
\def\II{{\mathbb{I}}}
\def\RR{{\mathbb R}}
\def\PP{{\mathbb P}}

\def\ZZ{{\mathbb Z}}
\def\NN{{\mathbb N}}

\def\vp{{\varphi}}

\newcommand{\tom}{{\tilde{\Omega} }}

\newcommand{\tmu}{{\tilde{\mu} }}

\newcommand{\tT}{{\tilde{T} }}
\newcommand{\tU}{{\tilde{U} }}

\def\bM{{\bf{M}}}

\providecommand{\floor}[1]{\left\lfloor \, #1 \, \right\rfloor}

\title[Escape rate and conditional escape rate]{Escape rate and conditional escape rate from a probabilistic point of view}

\author{C Davis}
\thanks{C Davis, Department of Mathematics, University of Oklahoma,
Norman, 73019. E-mail: {\tt {cdavis@math.ou.edu}}.} 
\author{N Haydn}
\thanks{N Haydn, Department of Mathematics, University of Southern California,
Los Angeles, 90089-2532. E-mail: {\tt {nhaydn@usc.edu}}.} 
\author{F Yang}
\thanks{F Yang, Department of Mathematics, University of Oklahoma,
Norman, 73019. E-mail: {\tt {fan.yang-2@ou.edu}}.} 
\date{\today}

\begin{document}

\begin{abstract}
We prove that for a sequence of nested sets $\{U_n\}$ with $\Lambda = \cap_n U_n$ a measure zero set, the localized escape rate converges to the extremal index of $\Lambda$, provided that the dynamical system is $\phi$-mixing at polynomial speed. We also establish the general equivalence between  the local escape rate for entry times and the local escape rate for returns.

\end{abstract}

\maketitle
\tableofcontents

\section{Introduction}

In recent years there has been an increasing interest in open dynamical systems,
which are dynamical systems with an invariant measure where one places a trap or
hole in the phase space, and looks at the decay rate of the measure of points that
are not caught by the trap up to some time (the survival set). This rate is known to be related to the rate of the correlations decay for the system (see~\cite{KL}). When the correlations decay exponentially fast, the decay rate for the measure of the survival set is typically exponential , and 
depends on the location and size of the trap. We invite the readers to the review article~\cite{DY} for a general overview on this topic. 

When the decay rate for the measure of the survival set is normalised by
the measure of the trap, one obtains the localized escape rate as the measure
of the trap goes to zero. Such problems are related to the entry times and return
times distribution but  this similarity does not allow to deduce 
limiting statistics from each other since the limits are taken in different ways, and the rate of convergence for the entry times to its limiting distribution is usually insufficient for the study of escape rates.

In the past local escape rates have been associated with either metic holes centred at 
a point whose radius decreased to zero or, in the presence of partitions, with 
cylinder sets which decrease to a single point. In this case a dichotomy has been 
established for many systems which shows the local escape rate to be equal
to one at non-periodic points and equal to the extremal index at periodic points. See the classical work~\cite{FP} for conformal attractors,~\cite{BDT,DT} for the transfer operator approach for interval maps, and~\cite{HY} for a probabilistic approach which applies to systems in higher dimension.
This mirrors the behaviour of the  limiting return times distributions that are 
Poisson at non-periodic points, and P\'olya-Aeppli compound Poisson
at periodic points in which case the compounded geometric distribution has the 
weights given by the extremal index $\theta \in (0,1)$. See~\cite{HV09}.

In this paper we will generalise the concept of localized escape rates to the cases when
the limiting set of the shrinking neighbourhoods are not any longer points, periodic or
non-periodic, but instead are allowed to be any null sets. 
Recently such a more general 
setting was used to obtain the limiting return times distribution at null sets 
which turn out to be compound Poisson in a more general sense where the 
associated parameters form a spectrum that is determined by the limiting cluster
size distribution (for this see below the coefficients $\alpha_k$). For more detail on this topic, see~\cite{HV19} and~\cite{FY19}. 

Unlike the conventional transfer operator method studied in~\cite{BDT} following the general setup in~\cite{KL} where exponential decay of correlations is required, our approach in this paper only assumes the system to be $\phi$-mixing at polynomial speed (surprisingly, this is enough to deduce that the escape rate is exponential!). This assumption may still not be optimal. However, we believe that the same results does not hold for $\alpha$-mixing systems in general. See the counter examples in~\cite{BDT} for systems modeled by Young's towers with polynomial tail, and note that such systems are $\alpha$-mixing at the same rate.

\section{Statement of results}\label{s.2}
Throughout this paper, we will assume that $(\bM, T, \cB, \mu)$ is a measure preserving system on some compact metric space $\bM$, with $\cB$ the Borel $\sigma$-algebra. Unless otherwise specified, $T$ is assumed to be non-invertible, although our result also holds in the invertible case, see Remark~\ref{r.1} and Theorem~\ref{t.leftmixing},~\ref{t.leftmixing.open} below. Let $\mathcal{A}$ be a measurable partition of 
$\bM$ (finite or countable). Denote, by $\mathcal{A}^n=\bigvee_{j=0}^{n-1}T^{-j}\mathcal{A},$  its $n$th join (in the invertible case, see Remark~\ref{r.1}). Then $\mathcal{A}^n$ is a partition of $\bM$ and its elements are
called $n$-cylinders. We assume that $\mathcal{A}$ is generating, that is $\bigcap_nA_n(x)$ consists of the 
singleton $\{x\}$.

\begin{definition} (i) The measure $\mu$ is {\em left $\phi$-mixing} with respect to  $\mathcal{A}$
if
$$
|\mu(A \cap T^{-n-k} B) - \mu(A)\mu(B)| \le \phi_L(k)\mu(A)
$$
for all $A \in \sigma(\mathcal{A}^n)$, $n\in\mathbb{N}$ and  $B \in \sigma(\bigcup_{j\ge1} \mathcal{A}^j )$,
where $\phi(k)$ is a decreasing function which converges to zero as $k\to\infty$. Here $\sigma(\mathcal{A}^n)$ is the $\sigma$-algebra generated by $n$-cylinders.\\
(ii) The measure $\mu$ is {\em right $\phi$-mixing} w.r.t.\  $\mathcal{A}$
if
$$
|\mu(A \cap T^{-n-k} B) - \mu(A)\mu(B)| \le \phi_R(k)\mu(B)
$$
for all $A \in \sigma(\mathcal{A}^n)$, $n\in\mathbb{N}$ and $B \in \sigma(\bigcup_j \mathcal{A}^j )$,
where $\phi(k)\searrow0$.\\
(iii) The measure $\mu$ is {\em $\psi$-mixing} w.r.t.\  $\mathcal{A}$
if
$$
|\mu(A \cap T^{-n-k} B) - \mu(A)\mu(B)| \le \psi(k)\mu(A)\mu(B)
$$
for all $A \in \sigma(\mathcal{A}^n)$, $n\in\mathbb{N}$ and $B \in \sigma(\bigcup_j \mathcal{A}^j )$,
where $\psi(k)\searrow0$. Clearly $\psi$-mixing implies both left and right $\phi$-mixing with $\phi(k)  = \psi(k)$.\\
\end{definition}

\begin{remark}
	If it is clear from context which type of mixing we are referring to (as is always the case in this paper), then we will suppress the subscripts $R,L.$
\end{remark}

We write, for any subset $U\subset \bM$, 
$$
\tau_U(x)=\min\{j\ge 1: T^j(x)\in U\}
$$
the first entry time to the set $U$. Then $\tau_U|_U$ is the first return time for points in $U$. 
We can now define the {\em escape rate} into $U$ by
$$
\rho(U)=\lim_{t\to\infty}\frac1t|\log\mathbb{P}(\tau_U>t)|
$$
whenever the limit exists. It captures the exponential decay rate for the set of points whose orbits have not visited $U$ before time $t$.
Observe that if $U\subset U'$ then $\mathbb{P}(\tau_{U'}>t)\le\mathbb{P}(\tau_{U}>t)$
and consequently $\rho(U)\le\rho(U')$. We define the {\em conditional escape rate} as
$$
\rho_U(U)=\lim_{t\to\infty}\frac1t|\log\mathbb{P}_U(\tau_U>t)|,
$$
where $\mathbb{P}_U$ is the conditioned measure on $U$.

We are particularly interested in the asymptotic behavior of $\rho(U)$ along a sequence of positive measure sets $\{U_n\}$ whose measure goes zero. For this purpose, we call $\{U_n\}$ a {\em nested sequence of sets} if $U_{n+1}\subset U_n$, and if $\lim_n\mu(U_n) = 0$. For the measure zero set $\Lambda = \cap_n U_n$, we define the {\em localized escape rate at $\Lambda$} as:
\begin{equation}\label{e.escaperate}
\rho(\Lambda, \{U_n\})=\lim_{n\to\infty}\frac{\rho(U_n)}{\mu(U_n)}.
\end{equation}
provided that the limit exists. We will show that under certain conditions, the localized escape rate of $\Lambda$ exists and does not depend on the choice of $U_n$. 

\subsection{Local escape rate for unions of cylinders}
First, we consider with the case where each $U_n$ is a union of $\kappa_n$-cylinders, for some non-decreasing sequence of integers $\{\kappa_n\}$.

We make some assumptions on the sizes of the nested sequence $\{U_n\}$.  For each $n$ and $j\ge 1$, we define $\cC_j(U_n) =\{A\in\cA^j,A\cap U_n\ne\emptyset\}$ the collection of all $j$-cylinders that have non-empty intersection with $U_n$.
Then, we write
$$
U_n^j = \bigcup_{A\in\cC_j(U_n)} A
$$
for the approximation of $U_n$ by $j$-cylinders from outside. For each fixed $j$, $\{U_n^j\}_n$ is also nested, that is, $U_{n+1}^j\subset U_n^j$. Obviously we have $U_n\subset U_n^j$ for all $j$, and $U_n=U_n^j$ if $j\ge \kappa_n$. 

\begin{definition}\label{d.2}
	A nested sequence $\{U_n\in\cA^{\kappa_n}\}$ is called a {\em good neighborhood system}, if: 
	\begin{enumerate}
		\item $\kappa_n\nearrow\infty$ and $\kappa_n\mu(U_n)^\varepsilon\to 0$ for some $\varepsilon\in(0,1)$;
		\item there exists $C>0$ and $p'>1$ such that $\mu(U^j_n)\le \mu(U_n) + Cj^{-p'}$ for all $j<\kappa_n$.
	\end{enumerate}
\end{definition}

\subsubsection{Local escape rate and the extremal index}
We make the following definitions, following ~\cite{HV19}

For a positive measure set $U\subset\bM$, we define the higher order entry times recursively:
$$
\tau^1_U=\tau_U, \mbox{ and }\tau^j_U(x) = \tau^{j-1}_U(x) + \tau_U(T^{\tau^{j-1}_U}(x)).
$$
For simplicity, we write $\tau^0_U = 0$ on $U$.

For a nested sequence $\{U_n\}$, define
\begin{equation}\label{e.hatalpha}
\hat\alpha_\ell(K,U_n)=\mu_{U_n}(\tau^{\ell-1}_{U_n}\le K),
\end{equation}
i.e. $\hat\alpha_\ell(K,U_n)$ is the conditional probability of having at least $(\ell-1)$ returns to $U_n$ before time $K$. We shall assume
that the limit $\hat\alpha_\ell(K)=\lim_{n\to\infty}\hat\alpha_\ell(K,U_n)$ exists for for $K\in \NN$ large enough
and every $\ell\ge 1$. By monotonicity the limits 
\begin{equation}\label{e.3}
\hat\alpha_\ell=\lim_{K\to\infty}\hat\alpha_\ell(K)
\end{equation} 
exist as $\hat\alpha_\ell(K)\le \hat\alpha_\ell(K')\le1$ for all $K\le K'$.
 Since we put $\tau^0_U = 0$, it follows that $\hat{\alpha}_1 = 1$.
We will see later that the existence of the limits defining $\hat\alpha_{\ell}$ implies the existence of the following limits:
\begin{equation}\label{e.4}
\alpha_1 = \lim_{K\to\infty}\lim_{n\to\infty}\mu_{U_n}(\tau_{U_n}>K) = 1-\hat\alpha_2.
\end{equation}
$\alpha_1\in[0,1]$ is generally known as the {\em extremal index} (EI). See the discussion in Freitas et al~\cite{FFT13}. 

The next theorem shows that the escape rate is indeed given by the extremal index.

\begin{main}\label{m.2}
Assume that $T:\bM\to \bM$ preserves a probability measure $\mu$ that is right $\phi$-mixing  with $\phi(k)\le Ck^{-p}$ for some  $C>0$ and $p>1$, and $\{U_n\}$ is a good neighborhood system such that $\{\hat{\alpha}_{\ell}\}$ defined in~\eqref{e.3} exists, and satisfies $\sum_{\ell}\ell\hat\alpha_{\ell}<\infty$. 

Then $\alpha_1$ defined by~\eqref{e.4} exists, 
and the localized escape rate at $\Lambda$ exists and satisfies
$$
\rho(\Lambda, \{U_n\}) = \alpha_1.
$$
\end{main}

\begin{remark}
Theorem~\ref{m.2} has a similar formulation for left $\phi$-mixing systems. See Remark~\ref{rem:bothsides} and Theorem~\ref{t.leftmixing}, \ref{t.leftmixing.open} for more detail.
\end{remark}

For Gibbs-Markov systems (for the precise definition, see Section~\ref{s.3}) the same result is true:

\begin{main}\label{m.gibbsmarkov}
	Assume that $T:\bM \to \bM$ is a Gibbs-Markov system with respect to the partition $\cA$. Let $\{U_n\}$ be a good neighborhood system such that $\{\hat{\alpha}_{\ell}\}$ defined in~\eqref{e.3} exists, and satisfies $\sum_{\ell}\ell\hat\alpha_{\ell}<\infty$. 
	
	Then $\alpha_1$ defined by~\eqref{e.4} exists. Furthermore, the localized escape rate at $\Lambda$ exists and satisfies
	$$
	\rho(\Lambda, \{U_n\}) = \alpha_1.
	$$
\end{main}

\begin{remark}\label{r.1}
	If $T$ is invertible, then one has to define the $n$-join by
	$$
	\mathcal{A}^n=\bigvee_{j=-n}^{n}T^{-j}\mathcal{A}. 
	$$
	In this case it is useful to write, for $m,n\in \ZZ$, $\cA^n_m: = \bigvee_{j=m}^{n}T^{-j}\mathcal{A}$. In particular we have  $\cA^n = \cA_{-n}^n$.
	The $\phi$ and $\psi$-mixing properties are defined by the same formulas. For integers $m,m',n,n'>0$,  if $U\in \cA^n_{-m}, V \in \cA^{n'}_{-m'}$ then for $k>n+m'$  we have $\mu(U\cap T^{-k} V) = \mu(T^{-m}U \cap T^{-k-m} V)$ where $T^{-m}U\in \cA_0^{m+n}$, $T^{-k-m}V \in \cA_{k+m-m'}^{k+m+n'}$. Note that $k+m-m'>n+m > 0$, so the estimate can be treated in the same way as the non-invertible case with only
	minor adjustments. However, the approximation $U^j_n$ need to be handled with care. See Remark~\ref{r.invertible} for a potential problem and the treatment. 
\end{remark}

\subsubsection{In the absence of short returns}
We will see below that when points in $\{U_n\}$ do not return to $U_n$ ``too soon'', then $\alpha_1 = 1$. To formulate this, we define the {\em periodic of $U$} as:
$$
\pi(U) = \min\{k>0: T^{-k}U\cap U\ne\emptyset\},
$$
and the {\em essential periodic of $U$} by:
$$
\pi_{\ess}(U) = \min\{k>0: \mu(T^{-k}U\cap U)>0\}.
$$
$\pi$ and $\pi_{\ess}$ mark the shortest return of points in $U$.
Note that $\pi(U)\le \pi_{\ess}(U)$ for all measurable $U\in\bM$, and equality holds if $T$ is continuous, $U$ is open and $\mu$ has full support.


\begin{maincor}\label{mc.0}
	Let $(\bM, T, \cB, \mu)$ be a measure preserving system. Assume that $\{U_n\}$ is a good neighborhood system with $\pi_{\ess}(U_n) \to \infty$, and $(T,\mu,\cA)$ satisfies one of the following two assumptions:
	\begin{enumerate}
		\item either $\mu$ is right $\phi$-mixing with $\phi(k)\le Ck^{-p}$ for some $p>1$;
		\item or $T$ is Gibbs-Markov;
	\end{enumerate}
	then the localized escape rate at $\Lambda$ exists and satisfies
	$$
	\rho(\Lambda, \{U_n\}) = 1.
	$$
\end{maincor}

As an immediate corollary, we have:

\begin{maincor}\label{mc.1}
	The conclusion of Corollary~\ref{mc.0} holds if the assumption ``$\pi_{\ess}(U_n)\to\infty$'' is replaced by the following assumptions:
	\begin{enumerate}
		\item $T$ is continuous,  $\Lambda = \cap_nU_n = \cap_n \overline{U}_n$;
		\item $\Lambda$ intersects every forward orbit at most once, that is, for every $x\in \Lambda$ we have $\Lambda\cap \{T^k(x) : k\ge 1\}=\emptyset$. 
	\end{enumerate}
\end{maincor}

\subsection{From cylinders to open sets: exceedance rate for the extreme value process}
Next, we deal with the case where  $\{U_n\}$ consists of open sets. For this purpose, we consider a continuous observable
$$\vp:\bM\to\RR \cup \{+\infty\},$$
such that the maximal value of $\vp$, which could be positive infinite, is achieved on a $\mu$  measure zero closed set $\Lambda$. Then we consider the process generated by the dynamics of $T$ and the observable $\varphi:$
$$
X_0=\varphi,\,X_1=\varphi\circ T,\,\ldots, X_k = \varphi\circ T^k
,\ldots.$$
Let $\{u_n\}$ be a non-decreasing sequence of real numbers. We will think of $u_n$ as a sequence of thresholds, and the event $\{X_k>u_n\}$
marks an exceedance above the threshold $u_n$. Also denote by $U_n$ the open set
\begin{equation}\label{e.Un}
U_n: = \{X_0>u_n\}.
\end{equation}
It is clear that $\{U_n\}$ is a nested sequence of sets. 

We are interested in the set of points where $X_k(x)$ remains under the threshold $u_n$ before time $t$. For this purpose, we put
$$
M_n = \max\{X_k: k=0,\ldots, n-1\},
$$
and
$$
\zeta(u_n) = \lim_{t\to\infty}\frac1t |\log \PP(M_t<u_n)|.
$$ 
Finally, define the {\em exceedance rate of $\vp$ along the thresholds $\{u_n\}$} as:
$$
\zeta(\vp, \{u_n\})=\lim_{n\to\infty}\frac{\zeta(u_n)}{\mu(U_n)}.
$$

We will make the following assumption on the shape of $U_n$. For each $r_n>0$, we approximate 
$U_n$ by two open sets (`o' and `i' stand for `outer' and `inner'):
$$
U^o_n 
=B_{r_n}(U_n)
= \bigcup_{x\in U_n} B_{r_n}(x), \mbox{\,\, and \,\,} 
U^i_n 
=U_n\setminus \overline{B_{r_n}(U_n)}
= U_n\setminus \left(\bigcup_{x\in \partial U_n}\overline{B_{r_n}(x)}\right).
$$
It is easy to see that 
$$
\overline{U^i_n}\subset U_n\mbox{ and } \overline{U_n}\subset  U^o_n,
$$

The following assumption requires $U_n$ to be {\em well approximable} by $U^{i/o}_n$. 
\begin{assumption}\label{a.1}
There exists a positive, decreasing sequence of real numbers $\{r_n\}$ with $r_n\to0$, such that
\begin{equation}\label{e.5}
\mu\left(U^o_n\setminus U^i_n\right) = o(1)\mu(U_n).
\end{equation}
\end{assumption}
Here $o(1)$ means the term goes to zero under the limit $n\to\infty$.

\begin{main}\label{m.3}
Assume that 
\begin{enumerate}
	\item either $\mu$ is right $\phi$-mixing with $\phi(k)\le Ck^{-p}$, $p>1$;
	\item or $(T,\mu,\cA)$ is a Gibbs-Markov system.
\end{enumerate}
Let $\vp:\bM\to \RR\cup\{+\infty\}$ be a  continuous function achieving its maximum on a measure zero set
 $\Lambda$. Let $\{u_n\}$ be a non-decreasing sequence of real numbers with $u_n\nearrow \sup \vp$, 
 such that the open sets $U_n$ defined by~\eqref{e.Un} satisfy Assumption~\ref{a.1} for a sequence
 $r_n$ that decreases to $0$. 
 Assume  $\{\hat{\alpha}_{\ell}\}$, defined by~\eqref{e.3}, exist and satisfy $\sum_{\ell}\ell\hat{\alpha}_{\ell}<\infty$.  
 Let $\kappa_n$ be the smallest positive integer for which $\diam\cA^{\kappa_n}\le r_n$ and  assume:
\begin{enumerate}[label=(\alph*)]
\item $\kappa_n\mu(U_n)^\varepsilon\to 0$ for some $\varepsilon \in (0,1)$;
\item $U_n$ has small boundary: there exist $C>0$ and $p'>1$, such that \\
$\mu\left(\bigcup_{A\in\cA^j, A\cap  B_{r_n}(\partial U_n) \ne\emptyset}A\right)\le C j^{-p'}$ for all $n$ and $j\le\kappa_n$.
\end{enumerate}
Then the exceedance rate of $\vp$ along $\{u_n\}$ exists and satisfies
$$
\zeta(\vp,\{u_n\}) = \alpha_1.
$$
\end{main}

\subsection{Conditional escape rate: a general theorem}
Our next theorem deals with the relation between the escape rate and the conditioned escape rate. 
\begin{main}\label{m.4}
For any measure preserving system $(\bM, T, \cB, \mu)$ and any positive measure set $U\in\bM$, we have 
$$
\rho(U) = \rho_U(U),
$$
assuming one of them exists.
\end{main}

Note that this theorem does not rely on the mixing assumption nor any information on the shape of $U$. In particular, if one defines the {\em localized conditional escape rate at $\Lambda$} as
$$
\rho_\Lambda(\Lambda, \{U_n\}) = \lim_{n\to\infty}\frac{\rho_{U_n}(U_n)}{\mu(U_n)},
$$
then we immediately have $\rho(\Lambda, \{U_n\}) = \rho_{\Lambda}(\Lambda, \{U_n\})-\alpha_1$ under the assumptions of Theorem~\ref{m.2} or~\ref{m.gibbsmarkov}.

\subsection{Escape rate on Young's towers with exponential tail}
Young's towers, also known as the Gibbs-Markov-Young structure, is first introduced by Young in~\cite{Y2} and~\cite{Y3}. Young's tower can be viewed as a discrete time suspension over a
Gibbs-Markov system $(\tom, \tT, \tmu)$, such that the roof function $R$ (in this case, it is usually
call the {\em return time function}) is integrable with respect to the measure $\tmu$. 

\begin{main}\label{m.5}
	Assume that $T$ is a $C^{2}$ map modeled by Young's tower, such that the return time function $R$ has exponential tail: there exists $\lambda\in (0,1)$ such that 
	$$
	\tmu(R>n) \lesssim \lambda^n.
	$$

	Let $\{U_n\subset \tom\}$ be a nested sequence of sets satisfying the assumption of Theorem~\ref{m.gibbsmarkov} in the cylinder case, or Theorem~\ref{m.3} in the open set case. Then the localized escape rate at $\Lambda=\cap_n U_n$ exists and satisfies 
	$$
	\rho(\Lambda,\{U_n\}) =\alpha_1.
	$$
\end{main}

\subsection{Organization of the paper}
This paper is organized in the following way. 
In section~\ref{s.3} we then state some properties that surround the parameters of 
very short returns whose presence is unaffected by the Kac timescale. In section~\ref{s.4}
we then prove the main results for cylinder approximations of the zero measure target set $\Gamma$.
 One crucial result here is Lemma~\ref{l.3} which yields the 
extremal index for the near zero time limiting distribution for entry times (as opposed 
to return times). Those results are then used in section~\ref{s.5} to extend them to 
the case when the approximating sets are metric neighbourhoods. In section~\ref{s.6} we 
then provide a general argument which shows that the local escape rate for entry 
times is the same as the local escape rate for returns. In section~\ref{s.7} we then 
show that the local escape rate persists for the induced map. Section~\ref{s.8}
is dedicated to entertain us with examples.

\section{Preliminaries}\label{s.3}
\subsection{Return and entry times along a nested sequence of sets}
In this section we recall the general results in~\cite{HV19} on the number of entries to an arbitrary null set $\Lambda$ within a cluster.

Given a sequence of nested sets $U_n,n=1,2,\ldots$ with $U_{n+1}\subset U_n$, $\cap_n U_n=\Lambda$ and $\mu(U_n)\to 0$, we will fix a large integer $K>0$ (which will be sent to infinity later), and assume that the limit 
$$
\hat\alpha_\ell(K) = \lim_{n\to\infty}\mu_{U_n}(\tau^{\ell-1}_{U_n}\le K)
$$
exists for $K$ sufficiently large and for every $\ell\in\NN$. By definition
$\hat\alpha_\ell(K)\ge \hat\alpha_{\ell+1}(K)$ for all $\ell$, and $\hat\alpha_1(K)=1$ due to our choice of $\tau^0 = 0$ on $U$. Also note that $\hat\alpha_\ell(K)$ is non-decreasing in $K$ for every $\ell$. As a result, we have for every $\ell\ge 1$:
\begin{equation}\label{e.3.0}
\hat\alpha_\ell = \lim_{K\to\infty}\hat\alpha_\ell(K) \mbox{ exists for every }\ell, \mbox{ and } \hat\alpha_1=1,\hat\alpha_\ell\ge \hat\alpha_{\ell+1}.
\end{equation}

Note that in the definition of $\hat{\alpha}$, the cut-off for the short return time $K$ does not depend on the set $U_n$. 
Another way to study the short return properties for the nested sequence $U_n$ is to look at
\begin{equation}\label{e.beta}
\hat\beta_\ell = \lim_{n\to\infty}\mu_{U_n}(\tau_{U_n}^{\ell-1}\le s_n)
\end{equation}
for some increasing sequence of integers $\{s_n\}$, with $s_n\mu(U_n)\to 0$ as $n\to\infty$. This is the approach taken by Freitas et al in~\cite{FFRS}. It is proven that for many systems (including Gibbs-Markov systems and Young's towers with polynomial tails), we have $\hat\beta_\ell=\hat{\alpha}_{\ell}$. See~\cite[Proposition 5.4 and 6.2]{FY19}

To demonstrate the power of desynchronizing $K$ from $n$, recall that for any set $U$, the  essential periodic of $U$ is given by:
$$
\pi_{\ess}(U) = \min\{k>0: \mu(T^{-k}U\cap U)>0\}.
$$
Then the following lemma can be easily verified using the definition of $\hat{\alpha}$:

\begin{lemma}\label{l.3.1}
Let $\{U_n\}$ be a sequence of nested sets. Assume that $\pi_{\ess}(U_n)\to \infty$ as $n\to\infty$, then $\hat{\alpha}_\ell$ exists and equals zero for all $\ell\ge 2$. 
\end{lemma}
\begin{proof}
For each $K$, one can take $n_0$ large enough such that $\pi_{\ess}(U_n)> K$ for all $n > n_0$. Then for $\ell \ge 2$, 
$$
\mu_{U_n}(\tau^{\ell-1}_{U_n}\le K)  \le \mu_{U_n} \left(\bigcup_{k=0}^K T^{-k}U_n\cap U_n\right) = 0
$$
since all the intersections have zero measure.
\end{proof}

Note that $\hat\alpha_\ell(K)$ is the conditional probability to have at least  $\ell-1$ returns in a cluster with length $K$. If we consider the level set:
\begin{equation}\label{e.alpha}
\alpha_\ell(K, U_n) = \mu_{U_n}(\tau^{\ell-1}_{U_n}\le K<\tau^\ell_{U_n}),
\end{equation}
and its limit
$$
\alpha_\ell(K)=\lim_{n\to\infty}\alpha_\ell(K, U_n),
$$
\begin{equation}\label{e.3.1}
\alpha_\ell = \lim_{K\to\infty}\alpha_\ell(K),
\end{equation}
then it is easy to see that 
$$
\alpha_\ell = \hat\alpha_\ell - \hat\alpha_{\ell+1},\mbox{ and so }\hat{\alpha}_\ell = \sum_{j\ge \ell}\alpha_j 
$$ 
which, in particular, implies the existence of $\alpha_\ell$. 

Next, following~\cite{HV19} we put for every integer $\ell >0$ and $K>0$, 
\begin{equation}\label{e.2}
\lambda_\ell(K,U_n) = \frac{\PP(\sum_{i=1}^{K}\II_{U_n}\circ f^i=\ell)}{\PP(\sum_{i=1}^{K}\II_{U_n}\circ f^i\ge 1)}.
\end{equation}
In other words, $\lambda_\ell(K,U_n)$ is, conditioned on having an entry to the set $U_n$, the probability to have precisely $\ell$ entries in a cluster with length $K$. The next theorem provides the relation between $\hat\alpha_\ell$ and $\lambda_\ell$:

\begin{theorem}\cite[Theorem 2]{HV19}\label{t.1}
	Assume that $U_n$ is a sequence of nested sets with $\mu(U_n)\to 0$. Assume that the limits in~\eqref{e.3.0} exist for $K$ large enough and every $\ell\ge 1$. Also assume that $\sum_{\ell=1}^{\infty}\ell\hat\alpha_\ell<\infty$.
	
	Then
	$$
	\lambda_\ell=\frac{\alpha_\ell - \alpha_{\ell+1}}{\alpha_1},
	$$ 
	where $\alpha_\ell = \hat\alpha_\ell -\hat\alpha_{\ell+1}$. In particular, the limit defining $\lambda_\ell$ exists. Moreover, the average length of the cluster of entries satisfies
	$$
	\sum_{\ell=1}^{\infty}\ell\lambda_\ell = \frac{1}{\alpha_1}.
	$$
\end{theorem}

For more properties on $\{\hat{\alpha}_\ell\}$, $\{\alpha_{\ell}\}$ and $\{\lambda_\ell\}$, we direct the readers to~\cite{HV19} and~\cite[Section 3]{FY19}.

\subsection{Gibbs-Markov systems}

A map $T:\bM\to\bM$ is called {\em Markov} if there is a countable measurable partition $\cA$ on $\bM$ with $\mu(A)>0$ for all $A\in \cA$, such that for all $A\in \cA$, $T(A)$ is injective and can be written as a union of elements in $\cA$. Write $\cA^n=\bigvee_{j=0}^{n-1}T^{-j}\cA$ as before, it is also assumed that $\cA$ is (one-sided) generating.

Fix any $\lambda\in(0,1)$ and define the metric $d_\lambda$ on $\bM$ by $d_\lambda(x,y) = \lambda^{s(x,y)}$, where $s(x,y)$ is the largest positive integer $n$ such that $x,y$ lie in the same $n$-cylinder. Define the Jacobian $g=JT^{-1}=\frac{d\mu}{d\mu\circ T}$ and $g_k = g\cdot g\circ T \cdots g\circ T^{k-1}$.

The map $T$ is called {\em Gibbs-Markov} if it preserves the measure $\mu$, and also satisfies the following two assumptions:\\
(i) The big image property: there exists $C>0$ such that $\mu(T(A))>C$ for all $A\in \cA$.\\
(ii) Distortion: $\log g|_A$ is Lipschitz for all $A\in\cA$.

In view of (i) and (ii), there exists a constant $D>1$ such that for all $x,y$ in the same $n$-cylinder, we have the following distortion bound:
$$
\left|\frac{g_n(x)}{g_n(y)}-1\right|\le D d_\lambda(T^nx,T^ny),
$$
and the Gibbs property:
$$
D^{-1}\le \frac{\mu(A_n(x))}{g_n(x)}\le D.
$$
It is well known (see, for example, Lemma 2.4(b) in~\cite{MN05}) that Gibbs-Markov systems are exponentially $\phi$-mixing, that is, $\phi(k)\lesssim \eta^k$ for some $\eta\in(0,1)$.\footnote{Here the notation $\lesssim$ means that the equality holds up to some constant $C>0$.}

\section{Escape rate for unions of cylinders}\label{s.4}
This section contains the proof of Theorem~\ref{m.2}, ~\ref{m.gibbsmarkov} and Corollary~\ref{mc.0},~\ref{mc.1}.  We will suppress the dependence of $\rho$ on $\{U_n\}$ and simply write $\rho(\Lambda)$  for the local escape rate at $\Lambda$.

\subsection{The block argument}
In this section, we will provide a general framework on the escape rate for polynomially $\phi$-mixing systems. The main lemma, which is Lemma~\ref{l.mainlemma}, allows us to reduce the escape rate (which is on the points that do not enter $U$ in a {\em large time-scale}) to the probability of having {\em short entries}.

First we introduce the following standard result for systems that are either left or right $\phi$-mixing. The proof can be found in~\cite{HY, A06}.

\begin{lemma}\cite[Lemma 4]{HY}
Assume that $\mu$ is either left or right $\phi$-mixing for the partition $\cA$. For $U\in\sigma(\cA^{\kappa_n})$, let $s,t>0$ and $\Delta<\frac{s}{2}$ then we have $$|\PP(\tau_U>s+t)-\PP(\tau_U>s)\PP(\tau_U>t)|\leq2(\Delta\mu(U)+\phi(\Delta-\kappa_n))\PP(\tau_U>t-\Delta).$$
\end{lemma}

Iterating the previous lemma, we obtain:
\begin{lemma}\label{l.blocks} 
Assume that $\mu$ is either left or right $\phi$-mixing for the partition $\cA$. Let $s>0$ and $\Delta<\frac{s}{2}$. Define $q=\floor{\frac{s}{\Delta}}$, $\eta=\frac{q}{q+1}$, and $\delta=2(\Delta\mu(U)+\phi(\Delta-\kappa_n)).$  Then for every $k\ge 3$ that is an integer multiple of $q^{-1},$ we have 
$$\PP(\tau_U>ks)\leq(\PP(\tau_U>s)+\delta^\eta)^{k-2}.
$$
\end{lemma}
\begin{proof}
We follow the proof of Theorem~1 in~\cite{HY} and observe that by the previous lemma, 
for $k\in[1,3]$ that is an integer multiple of $q^{-1}$, we have
\begin{align*}
\PP(\tau_U>ks)&\leq\PP(\tau_U>s)\PP(\tau_U>(k-1)s)+\delta\PP(\tau_U>(k-1-q^{-1})s)\\
&\leq\PP(\tau_U>s)^2+\delta\PP(\tau_U>s)\\
&=\PP(\tau_U>s)(\PP(\tau_U>s)+\delta)\\
&\leq(\PP(\tau_U>s)+\delta^\eta)^{k-2},
\end{align*}
where we used $\PP(\tau_U>s)\leq\PP(\tau_U>s)+\delta^\eta$, and $\PP(\tau_U>s)\leq 1.$ 
 
For $k>3$, we use induction on $m=k\cdot q\in\NN$:
\begin{align*}
\PP(\tau_U>ks)&\leq\PP(\tau_U>s)\PP(\tau_U>(k-1)s)+\delta\PP(\tau_U>(k-1-q^{-1})s)\\
&\leq\PP(\tau_U>s)(\PP(\tau_U>s)+\delta^\eta)^{k-3}+\delta(\PP(\tau_U>s)+\delta^\eta)^{k-3-q^{-1}}\\
&=(\PP(\tau_U>s)+\delta^\eta)^{k-3-q^{-1}}[\PP(\tau_U>s)(\PP(\tau_U>s)+\delta)^{q^{-1}}+\delta]\\
&\leq(\PP(\tau_U>s)+\delta^\eta)^{k-2}.
\end{align*}
We justify the last inequality as follows. By definition of $\eta,$ we have $\delta=\delta^\eta\delta^{\frac{\eta}{q}}\leq\delta^\eta(\PP(\tau_U>s)+\delta^\eta)^{q^{-1}}.$
Consider the bracketed term in the third line. \begin{align*}&\PP(\tau_U>s)(\PP(\tau_U>s)+\delta)^{q^{-1}}+\delta\\\leq&\PP(\tau_U>s)(\PP(\tau_U>s)+\delta)^{q^{-1}}+\delta^\eta(\PP(\tau_U>s)+\delta^\eta)^{q^{-1}}\\
=&\PP(\tau_U>s)+\delta^\eta)^{1+q^{-1}}.\end{align*}
By induction this completes the proof.
\end{proof}

The next lemma establishes the relation between the escape rate and the probability of short entries:

\begin{lemma}\label{l.mainlemma}
Assume that $\mu$ is either left or right $\phi$-mixing for the partition $\cA$, with $\phi(k)\le Ck^{-p}$ for some $p>0$. Let $\{U_n\in\sigma(\cA^{\kappa_n})\}$ be a nested sequence of sets for some $\kappa_n\nearrow\infty$. Furthermore, assume that there exists $\varepsilon\in(0,1)$, such that $\kappa_n\mu(U_n)^\varepsilon \to 0$.

Then we have 
\begin{equation}\label{e.mainlemma}
\rho(\Lambda) = \lim_{n\to\infty}\frac{\PP(\tau_{U_n}\le s_n)}{s_n\mu(U_n)},
\end{equation}
where $s_n= \floor{\mu(U_n)^{-(1-a)}}$ for any fixed $a>0$ small enough.
\end{lemma}

\begin{remark}
At first glance, the RHS of~\eqref{e.mainlemma} is similar to the definition of the local escape rate in~\eqref{e.escaperate}. However, since $s_n \ll \mu(U_n)^{-1}$ (where the latter is the average return time given by Kac's formula), $\PP(\tau_{U_n}\le s_n)$ concerns the probability of short entries to $U$.
\end{remark}

\begin{proof}
Let $\{s_n\}, \{\Delta_n\}$ be increasing sequences of positive integers with $\Delta_n<s_n/2$, whose choice will be specified later. Write $q_n = \floor{s_n/\Delta_n}$, $\eta_n = \frac{q_n}{q_n+1}$ and $\delta_n = 2(\Delta_n\mu(U_n) + \phi(\Delta_n-\kappa_n))$ as before. We again follow largely the proof of Theorem~1 in~\cite{HY} and get by Lemma~\ref{l.blocks}
$$
\frac{1}{ks_n} |\log\PP(\tau_{U_n}>ks_n)| = \frac{k-2}{ks_n}|\log\left(\PP(\tau_{U_n}>s_n)+\delta_n^{\eta_n}\right)|.
$$
Take limit as $k\to\infty$ and note that $\PP(\tau_{U_n}>s_n) = 1-\PP(\tau_{U_n}\le s_n)$, we obtain
\begin{align*}
\rho(U_n) =& \lim_{k\to\infty}\frac{1}{ks_n} |\log\PP(\tau_{U_n}>ks_n)|\\
 = &\frac{1}{s_n}\big(\PP(\tau_{U_n}\le s_n)+o(s_n\mu(U_n))+\mathcal{O}(\delta_n^{\eta_n})\big).\numberthis\label{e.8}
\end{align*}
Here we used the trivial estimate 
$$
\PP(\tau_{U_n}\le s_n) \le \PP\!\left(\bigcup_{1\le k\le s_n}T^{-k}(U_n)\right)\le s_n\mu(U_n).
$$

Divide~\eqref{e.8} by $\mu(U_n)$ and let $n\to\infty$, we obtain
\begin{equation}\label{e.9}
\rho(\Lambda) = \lim_{n\to\infty}\left(\frac{\PP(\tau_{U_n}\le s_n)}{s_n\mu(U_n)} + \frac{\delta_n^{\eta_n}}{s_n\mu(U_n)}\right).
\end{equation}

It remains to show that the second term converges to zero for some proper choice of $\{s_n\}$ and $\Delta_n$. For this purpose, we fix some $a\in(0,1), b\in(\varepsilon,1)$ and choose $s_n= \floor{\mu(U_n)^{-(1-a)}}$, and $\Delta_n = \floor{\mu(U_n)^{-b}}\gg\kappa_n = o(\mu(U_n)^{-\varepsilon})$. Then we have ($a_n \lesssim b_n$ means there exists constant $C$ such that $a_n\le C\cdot b_n$):
\begin{align*}
\frac{\delta_n^{\eta_n}}{s_n\mu(U_n)}\lesssim & \frac{\Delta_n^{\eta_n}\mu(U_n)^{\eta_n}}{s_n\mu(U_n)}  + \frac{\phi(\Delta_n-\kappa_n)^{\eta_n}}{s_n\mu(U_n)}\\
\le & \Delta_n\mu(U_n)^{\eta_n-a} + \Delta_n^{-p\eta_n}\mu(U_n)^{-a}\\
\le & \mu(U_n)^{\eta_n-a-b} + \mu(U_n)^{bp\eta_n-a}.
\end{align*} 

In order for both terms to go to zero, we need:
\begin{enumerate}
\item $1-a>b$, which guarantees that $s_n\gg \Delta_n$, so $q_n\to \infty$ and consequently $\eta_n\nearrow 1$; then the first term will go to zero;
\item $bp>a$, so that the second term goes to zero.
\end{enumerate}
Both requirements are satisfied if we take any $b\in(\varepsilon,1)$, then choose  $0<a<\min\{1-b,bp\}$. Combine this with~\eqref{e.9} we conclude that 
$$
\rho(\Lambda) = \lim_{n\to\infty}\frac{\PP(\tau_{U_n}\le s_n)}{s_n\mu(U_n)},
$$
as desired.
\end{proof}

In the remaining part of this section, we will prove that the RHS of~\eqref{e.mainlemma} coincides with the extreme index defined by~\eqref{e.4}. But before we move on, let us state a direct corollary of the previous lemma, which is interesting in its own right.
\begin{proposition}\label{p.coarsebound}
Assume that $\mu$ is either left or right $\phi$-mixing for the partition $\cA$, with $\phi(k)\le Ck^{-p}$ for some $p>0$. Let $\{U_n\in\sigma(\cA^{\kappa_n})\}$ be a nested sequence of sets for some $\kappa_n\nearrow\infty$. Furthermore, assume that there exists $\varepsilon\in(0,1)$, such that $\kappa_n\mu(U_n)^\varepsilon \to 0$.

Then we have 
\begin{equation}
\rho(\Lambda) \in [0,1],
\end{equation}
provided that the local escape rate at $\Lambda$ exists.
\end{proposition}

\begin{proof}
The lower bound is clear. For the upper bound, the trivial estimated which is already used in the proof of Lemma~\ref{l.mainlemma} yields 
$$
\frac{\PP(\tau_{U_n}\le s_n)}{s_n\mu(U_n)} \le \frac{s_n\mu(U_n)}{s_n\mu(U_n)}=1.
$$
\end{proof}

\subsection{Proof of Theorem~\ref{m.2} and~\ref{m.gibbsmarkov}}\label{s.4.2}
First we prove Theorem~\ref{m.2} using the following lemma, which is stated for right $\phi$-mixing systems. The proof can be adapted for left $\phi$-mixing systems as well, with certain modification on the assumptions of $U_n$ (in particular, on how $U_n$ can be approximated by shorter cylinders). See Remark~\ref{rem:bothsides} below and the discussion in Section~\ref{s.4.3}.

\begin{lemma}\label{l.3} Let $\mu$ be right $\phi$-mixing for the partition $\cA$, with $\phi(k) \le Ck^{-p}$ for some $p>1$. Assume that $\{U_n\}$ is a good neighborhood system, such that $\hat\alpha_\ell(K)$ exists for $K$ large enough, and $\sum_{\ell}\hat{\alpha}_\ell<\infty$. Then we have 
$$
\lim_{n\to\infty}\frac{\PP(\tau_{U_n}\le s_n)}{s_n\mu(U_n)}=\alpha_1
$$
for any increasing sequence  $\{s_n\}$ for which $s_n\mu(U_n)\to0$ as $n\to\infty$.
\end{lemma}

\begin{proof} 
For an given integer $s$, write $Z_n^{s} = \sum_{j=1}^s\mathbb{I}_{U_n}\circ T^j$ which counts the number of entries to $U_n$ before time $s$.
Let $K$ be a large integer, then  by~\cite{HV19} Lemma~3 for every $\varepsilon>0$ one has
$\mathbb{P}(\tau_{U_n}\le K)=\alpha_1K\mu(U_n)(1+\mathcal{O}^*(\varepsilon))$ for all $n$ large enough,
where the notation $\mathcal{O}^*$ means that the implied constant is one (i.e.\ $x=\mathcal{O}^*(\varepsilon)$
if $|x|< \varepsilon$).
For simplicity, assume $r=s_n/K$ is an integer and  put
$$
V_q=\{Z_n^K\circ T^{qK}\ge1\},
$$
$q=0,1,\dots,r-1$, and 
$$
D_q=\{V_q,Z_n^{(r-q-1)K}\circ T^{(q+1)K}=0\}.
$$
Then
$$
\{Z_n^{s_n}\ge1\}=\bigcup_{q=0}^{r-1}D_q
$$
is a disjoint union.
Let us now estimate
\begin{align*}
\mathbb{P}(Z_n^{(r-q-1)K}\circ T^{(q+1)K}\ge1, V_q)\hspace{-3cm}&\\
&\le \mathbb{P}(Z_n^{(r-q-1)K-2\sqrt{K}}\circ T^{(q+1)K+2\sqrt{K}}\ge1, V_q)+2\sqrt{K}\mu(U_n)\\
&\le2\sqrt{K}\mu(U_n)+\mu(V_q, Z_n^{s_n - (q+1)K - 2\kappa_n}\circ T^{(q+1)K+2\kappa_n} \ge 1)\numberthis\label{e.III}\\
&\hspace{3cm}+\sum_{j=qK}^{(q+1)K-1}\sum_{i=(q+1)K+2\sqrt{K}}^{(q+1)K+2\kappa_n}\mu(T^{-j}U_n\cap T^{-i}U_n)\\
&=: {\rm I}+{\rm II}+{\rm III}.
\end{align*}
To bound ${\rm II}$, note that $\{Z_n^{s_n - (q+1)K - 2\kappa_n}\circ T^{(q+1)K+2\kappa_n}\ge 1\}$ is the event of having a hit between $[(q+1)K+2\kappa_n, s_n]$. We cut this interval into $t_n = \floor{\frac{s_n - (q+1)K - 2\kappa_n}{K}} \ge 0$ (II is void when $t_n$ is negative) many blocks with length $K$. This allow us to estimate: 
\begin{align*}
{\rm II}\le& \sum_{j=0}^{t_n+1} \mu(Z_n^K\circ T^{qK}\ge 1, Z_n^K\circ T^{(q+1+j)K+2\kappa_n}\ge 1) \\
 =& \sum_{j=0}^{t_n+1}\sum_{k = 1}^K \mu(T^{-qK-k}U_n, Z_n^K\circ T^{(q+1+j)K+2\kappa_n}\ge 1) \\
 \le& \sum_{j=0}^{t_n+1}\sum_{k = 1}^K\mu(Z_n^K\ge 1)(\mu(U_n) + \phi((j+1)K+\kappa_n -k ))\\
 \le& \sum_{i=\kappa_n}^{s_n+K} \mu(V_q)(\mu(U_n) + \phi(i)),
\end{align*}
where we used $(t_n +1)$ many blocks instead of $t_n$ to cover the remaining $\le K$ many hits at the end. The third inequality follows from right $\phi$-mixing, and the last line is due to $\mu(Z_n^K\ge 1) = \mu(V_q)$.

For the third term in~\eqref{e.III}, we use right $\phi$-mixing again to get (and recall that $U_n^j$ is the outer-approximation of $U_n$ by $j$-cylinders):
\begin{align*}
{\rm III}\le & \sum_{j=qK}^{(q+1)K-1}\sum_{i=(p+1)K+2\sqrt{K}}^{(q+1)K+2\kappa_n}\mu(U_n\cap T^{-(i-j)}U_n)\\
\le & K\sum_{j=2\sqrt{K}}^{2\kappa_n}\mu(U_n^{j/2}\cap T^{-j}U_n)\\
\le & K\sum_{j=2\sqrt{K}}^{2\kappa_n} \mu(U_n)(\mu(U_n^{j/2}) + \phi(j/2))\\
= & \cO(1)\mu(V_q)\sum_{j=2\sqrt{K}}^{2\kappa_n} (\mu(U_n^{j/2}) + \phi(j/2)),
\end{align*}
where the last equality follows from 
$$
\mu(V_q) = \PP(\tau_{U_n}\le K) = \alpha_1K\mu(U_n)(1+\cO^*(\varepsilon)).
$$

Combining the previous estimates, we get
\begin{eqnarray*}
\mathbb{P}(Z_n^{(r-q-1)K}\circ T^{(q+1)K}\ge1, V_q)\hspace{-4cm}&&\\
&\le& \mathbb{P}(Z_n^{(r-q-1)K-2\sqrt{K}}\circ T^{(q+1)K+2\sqrt{K}}\ge1, V_q)+2\sqrt{K}\mu(U_n)\\
&\le&2\sqrt{K}\mu(U_n)
+\mu(V_q)\sum_{i=\kappa_n}^{s_n+K}(\mu(U_n)+\phi(i))\\
 & &\hspace{3cm}+\mu(V_q)\cO(1)\sum_{j=\sqrt{K}}^{\kappa_n}(\mu(U_n^{j})+\phi(j))\\
&\le&\mu(V_q)F,
\end{eqnarray*}
where
$$
F=\frac2{\sqrt{K}}+(s_n+K)\mu(U_n)+\cO(1)(\phi^1(\sqrt{K})
+\sum_{j=\sqrt{K}}^{\kappa_n}\mu(U_n^{i}))
$$
and $\phi^1(u)=\sum_{i=u}^\infty\phi(i)$ is the tail-sum of $\phi$ which by assumption goes to zero
as $u$ goes to infinity.

If $n$ is large enough so that $\max\{s_n\mu(U_n), \kappa_n\mu(U_n), \phi^1(\kappa_n)\}<\varepsilon$ then
\begin{align*}
F&\le2\varepsilon+\frac2{\sqrt{K}}+\cO(1)\left(\phi^1(\sqrt{K})
+\kappa_n\mu(U_n)+\sum_{i=\sqrt{K}}^{\kappa_n}i^{-p'}\right)\\
&\lesssim\varepsilon+\frac1{\sqrt{K}}+\phi^1(\sqrt{K})
+K^{-\frac{p'-1}2},
\end{align*}
where we used the assumption that $\mu(U_n^{i})\le\mu(U_n)+ Ci^{-p'}$ for some $p'>1$.
Consequently
$$
\mu(D_q)=\mu(V_q)-\mathbb{P}(V_q,Z_n^{(r-q-1)K}\circ T^{(q+1)K}\ge1)
=\mu(V_q)(1+\mathcal{O}^*(F)),
$$
and since $\{Z_n^{qK}\ge1,V_q\}=V_q$ and $\mu(V_q)=\mu(V_0)$ we get
$$
\mathbb{P}(Z_n^{s_n}\ge1)
=\sum_{q=0}^{r-1}\mathbb{P}(D_q)
=r\mu(V_0)(1+\mathcal{O}^*(F)).
$$
Since by~\cite{HV19} Lemma~3 $\mu(V_0)=\alpha_1K\mu(U_n)(1+\mathcal{O}^*(\varepsilon))$
we obtain
$$
\mathbb{P}(\tau_{U_n}\le s_n)
=r\mu(V_0)(1+\mathcal{O}^*(F))
=\alpha_1s_n\mu(U_n)(1+\mathcal{O}^*(\varepsilon+F)).
$$
The statement of the lemma now follows if we let $\varepsilon\to0$ and then $K\to\infty$.
\end{proof}

\begin{remark}\label{rem:bothsides}
Similar to the previous lemmas which hold for both left and right $\phi$-mixing measures, Lemma~\ref{l.3} has a similar formulation in the left $\phi$-mixing case. The estimate of II in~\eqref{e.III} is mostly the same (see the proof of Lemma~\ref{l.3'} below for more detail). However, this would require us to modify the definition of the approximated sets 
$U^i_n$ as 
$$
\tilde{U}_n^i=T^{-(n-i)}A_i(T^{n-i}U_n),
$$
with the assumption that the measure of $\tilde{U}^i_n$ is small (preferably summable in $i$, similar to (2) in Definition~\ref{d.2}). This is indeed the treatment in~\cite[Lemma 3]{HY} when $\Lambda = \{x\}$. However, such an assumption may not hold when $\Lambda$ is a non-singleton null set. The right $\phi$-mixing property avoids this problem. 
\end{remark}

\begin{remark}\label{r.invertible}
So far we have assumed that $T$ is non-invertible. This is because in the invertible case, the approximation $U^j_n$ and $\tilde U^j_n$ may become the entire space.
 As an example, take $\bM=\Omega$ to be a full, two-sided shift space and $T=\sigma$ the left-shift. Let the sets $U_n$ be $n$-approximation of an unstable leaf $\Gamma$ through a non-periodic point $x\in\Omega$,
e.g.\ $\Gamma=\{y\in\Omega: y_i=x_i\;\forall\;i\le0\}$. Obviously $\Gamma$ is a null set
but in this case we get that $\tilde{U}^j=\Omega$ the entire space whenever $i<n/2$. For a geometric example,let $T$ be an Anosov diffeomorphisms on $\TT^n$ with minimal unstable foliations and $\Lambda$ be the local unstable manifold at some $x\in \bM$. Then $T^j \Lambda$ eventually becomes $\varepsilon$-dense in $\bM$, and the approximation $\tilde{U}^i_n$ (with respect to a Markov partition $\cA$) is the entire space for $i$ small. By symmetry and Remark~\ref{r.invertible}, we see that if $\Lambda$ is chosen to be a local stable manifold then $U^j_n = \bM$ for $j$ small.  

On the other hand, in the proof of Lemma~\ref{l.3} the approximation $U^j_n$ is only used to control III of~\eqref{e.III}. Later this observation will allow us to obtain a result for invertible systems where this term does not appear. See Theorem~\ref{t.leftmixing} and~\ref{t.leftmixing.open} below.

\end{remark}

Below we state an alternate version of Lemma~\ref{l.3} where the right $\phi$-mixing assumption is replaced by the Gibbs-Markov property. This allows us to bypass the issue stated in Remark~\ref{rem:bothsides} and keep the choice of $U_n^i$. 

\begin{lemma}\label{l.3'}
Let $(T,\mu,\cA)$ be a Gibbs-Markov system. Assume that $\{U_n\}$ is a good neighborhood system, such that $\hat\alpha_\ell(K)$ exists for $K$ large enough, and $\sum_{\ell}\hat{\alpha}_\ell<\infty$. Then we have 
$$
\lim_{n\to\infty}\frac{\PP(\tau_{U_n}\le s_n)}{s_n\mu(U_n)}=\alpha_1
$$
for any increasing sequence  $\{s_n\}$ for which $s_n\mu(U_n)\to0$ as $n\to\infty$.	
\end{lemma}

\begin{proof}
Recall that Gibbs-Markov systems are left $\phi$-mixing with exponential rate. The proof follows the lines of Lemma~\ref{l.3} up to equation~\eqref{e.III}, which is now estimated using the left $\phi$-mixing as:
\begin{align*}
{\rm II} = &\ \mu(V_q, Z_n^{s_n - (q+1)K - 2\kappa_n}\circ T^{(q+1)K+2\kappa_n} \ge 1)\\
\le& \sum_{i=(q+1)K+2\kappa_n}^{s_n}\mu(V_q \cap T^{-i}U_n)\\
\le& \sum_{i=\kappa_n}^{s_n}\mu(V_q)(\mu(U_n) + \phi(i)).
\end{align*}
Note that the proof in this case is much short and the bound is almost the same as before.

For III, we first split the term into the summation over the intersections of $U_n$ with $T^iU_n$:

\begin{align*}
{\rm III}\le & \sum_{j=qK}^{(q+1)K-1}\sum_{i=(p+1)K+2\sqrt{K}}^{(q+1)K+2\kappa_n}\mu(U_n\cap T^{-(i-j)}U_n)\\
\le & \ K\sum_{j=2\sqrt{K}}^{2\kappa_n}\mu(U_n\cap T^{-j}U_n).
\end{align*}
Each term in the summation can be bounded by:
\begin{align*}
\mu(U_n\cap T^{-j}U_n)\le& \sum_{A\in\cC_j(U_n)}\mu(T^{-j}U_n\cap A)\\
=&\sum_{A\in\cC_j(U_n)}\frac{\mu(T^{-j}U_n\cap A)}{\mu(A)}\mu(A)\\
\lesssim&\sum_{A\in\cC_j(U_n)}\frac{\mu(T^j(T^{-j}U_n\cap A))}{\mu(T^jA)}\mu(A)\\
\lesssim&\sum_{A\in\cC_j(U_n)}\mu(U_n)\mu(A)\\
=&\mu(U_n)\mu\left(\bigcup_{A\in\cC_j(U_n)}A\right)=\mu(U_n)\mu(U_n^j),
\end{align*}
where the third and forth inequality follow from the distortion and the big image property of Gibbs-Markov systems. See~\cite[Theorem D]{FY19}.

Then
$$
{\rm III} \le c_1K\mu(U_n)\sum_{j= 2\sqrt{K}}^{\kappa_n} \mu(U_n^j) = \cO(1)\mu(V_p)\sum_{j= 2\sqrt{K}}^{2\kappa_n} \mu(U_n^j),
$$
for some $c_1$
and the rest of the proof is identical to Lemma~\ref{l.3}.
\end{proof}

Now Theorem~\ref{m.2} and~\ref{m.gibbsmarkov} are  immediate consequences of  Lemma~\ref{l.mainlemma},~\ref{l.3} and~\ref{l.3'}.

\begin{proof}[Proof of Corollary~\ref{mc.0}]
This corollary directly follows from Lemma~\ref{l.3.1}.
\end{proof}

\begin{proof}[Proof of Corollary~\ref{mc.1}]
We need the following proposition from~\cite{FY19}:
\begin{proposition}\cite[Proposition 6.3]{FY19}
Let $T$ be a continuous map on the compact metric space $\bM$, and $\{U_n\}$ a nested sequence of sets such that $\cap_nU_n = \cap_n \overline{U}_n$.  Then $\pi(U_n) \to\infty$ if and only if $\Lambda=\cap_nU_n $ intersects every forward orbit at most once.
\end{proposition}

Since $\pi_{\ess}(U)\ge \pi(U)$, we have $\pi_{\ess}(U_n)\to\infty$. Combined with Corollary~\ref{mc.0}, we obtain Corollary~\ref{mc.1}.
\end{proof}

\subsection{Some remarks on the extremal index}\label{s.4.3}
In the classic literature (for example,~\cite{FFT13},~\cite{FFFH} and~\cite{FFRS}), the extremal index is defined as
\begin{equation}\label{e.extremal}
\theta = \lim_{n\to\infty}\mu_{U_n}(\tau_{U_n}>K_n),
\end{equation}
where $K_n\to\infty$ is some increasing sequence of integers. It is shown in~\cite[Proposition 5.4]{FY19} that under the assumption of Theorem~\ref{m.gibbsmarkov} we have 
$$
\alpha_1 = \theta.
$$

It is also straight forward to check that the proof of Lemma~\ref{l.3} and~\ref{l.3'} remain true with $\alpha_1$ replaced by $\theta$. We state this as the following proposition:

\begin{proposition}\label{p.4.9}
	Assume that one of the following assumptions holds:
	\begin{enumerate}
		\item either $\mu$ is right $\phi$-mixing with $\phi(k)\lesssim k^{-p}$, $p>1$;
		\item or $(T,\mu,\cA)$ is a Gibbs-Markov system.
	\end{enumerate}
	Let $\theta$ be the extremal index defined by~\eqref{e.extremal} for some sequence $\{K_n\}$. 
	Then for any good neighborhood system $\{U_n\}$ and any increasing sequence $\{s_n\}$ with $s_n\mu(U_n)\to 0$ and $s_n/K_n\to\infty$, we have 
	$$
	\lim_{n\to\infty}\frac{\PP(\tau_{U_n}\le s_n)}{s_n\mu(U_n)}=\theta.
	$$
	Furthermore, the local escape rate at $\Lambda = \bigcap_n U_n$ exists and satisfies 
	$$
	\rho(\Lambda) = \theta.
	$$
\end{proposition}

Note that in the proof of Lemma~\ref{l.3}, the bound on II of~\eqref{e.III} holds for both left and right $\phi$-mixing systems, as already demonstrated in Lemma~\ref{l.3'}. On the other hand, for $\theta$ defined 
by~\eqref{e.extremal}, III of~\eqref{e.III} does not exist when $K_n > \kappa_n^2$. By Remark~\ref{rem:bothsides} and~\ref{r.invertible}, we can drop 
the right $\phi$-mixing and the Gibbs-Markov assumption, obtaining the following theorem for left $\phi$-mixing systems that is either invertible or non-invertible:

\begin{theorem}\label{t.leftmixing}
Assume that $T:\bM\to \bM$ is a dynamical system, either invertible or non-invertible, and preserves a measure $\mu$ that is left $\phi$-mixing  with $\phi(k)\le Ck^{p}$ for some  $C>0$ and $p>1$. Let $\{U_n\in\cA^{\kappa_n}\}$ be a nested sequence of sets with $\kappa_n \mu(U_n)^\varepsilon\to 0$ for some $\varepsilon\in(0,1)$. 

Assume that $\theta$ defined by~\eqref{e.extremal} exists for some sequence $\{K_n\}$ with $K_n > \kappa_n^2$. Then the localized escape rate at $\Lambda$ exists and satisfies
$$
\rho(\Lambda) = \theta.
$$
\end{theorem}

\section{Escape rate for open sets: an approximation argument}\label{s.5}

First, observe that 
$$
\{M_t<u_n\} = \{\tau_{U_n} > t\}.
$$
As a result, we have 
$$
\zeta(u_n) = \lim_{t\to\infty}\frac1t |\log \PP(M_t < u_n)| = \lim_{t\to\infty}\frac1t|\log \PP(\tau_{U_n} > t)| = \rho(U_n), 
$$
therefore we have
$$
\zeta(\varphi, \{u_n\}) = \rho(\Lambda, \{U_n\}).
$$

The following proposition allows us to replace $\{U_n\}$ by its cylinder-approximation.
\begin{proposition}\label{p.approx}
	Let $\{U_n\}$, $\{V_n\}$ and $\{W_n\}$ be sequences of nested sets with $V_n\subset U_n\subset W_n$ for each $n$, and $\Lambda = \bigcap_n U_n = \bigcap_n V_n = \cap_n W_n$. Assume that 
	\begin{equation}\label{e.approx}
		\mu(W_n\setminus V_n) = o(1)\mu(V_n),
	\end{equation}
	and 
	$\rho(\Lambda, \{W_n\})=\rho(\Lambda, \{V_n\}) = \alpha$.
	
	Then
	$$
	\rho(\Lambda, \{U_n\}) =\alpha.
	$$
\end{proposition}

\begin{proof}
	$V_n\subset U_n\subset W_n$ implies that $\tau_{W_n}\ge \tau_{V_n} \ge \tau_{U_n}$ and therefore
	$$
	\rho(W_n)\ge \rho(U_n) \ge \rho(V_n). 
	$$ 
	On the other hand,~\eqref{e.approx} means that $\mu(W_n)/\mu(V_n)\to 1$. We thus obtain
	$$
	\rho(\Lambda, \{W_n\})\ge \rho(\Lambda, \{U_n\}) \ge \rho(\Lambda, \{V_n\}),
	$$
	and the proposition follows from the squeeze theorem.
\end{proof}

\begin{proof}[Proof of Theorem~\ref{m.3}]
	For the sequence $\{r_n\}$ given in Assumption~\ref{a.1}, we write $\kappa_n$ for the smallest integer such that $\diam \cA^{\kappa_n} \le r_n$. Then consider 
	$$
	V_n = \cup_{A\in\cA^{\kappa_n}, A\subset U_n} A, \hspace{2cm} 	W_n = \cup_{A\in\cA^{\kappa_n}, A\cap  U_n\ne\emptyset} A.
	$$
	Clearly we have $V_n\subset U_n\subset W_n$ for each $n$. Moreover, the choice of $\kappa_n$ gives
	$$
	U^i_n\subset V_n, \hspace{2cm} W_n\subset U^o_n.
	$$
	Combine this with~\eqref{e.5}, we have $\mu(W_n\setminus V_n) = o(1)\mu(V_n).$
	
	Let us write $\hat\alpha_\ell^*$, $* = U,V,W$ for $\hat\alpha_\ell$ defined using $\{U_n\},\{V_n\},\{W_n\}$, respectively. Then it is proven in~\cite[Lemma 5.6]{FY19} that 
	$$
	\hat\alpha_\ell^V = \hat\alpha_\ell^U =\hat\alpha_\ell^W. 
	$$
	In particular, $\sum_\ell \ell\hat\alpha_\ell^U<\infty$ implies that the same holds for $\hat\alpha_\ell^*$, $* = V, W$, and the value of $\alpha_1$ defined by $\{V_n\}, \{U_n\}, \{W_n\}$ are equal.
	
	It remains to show that $\{V_n\}$ and $\{W_n\}$ are good neighborhood systems. (1) of Definition~\ref{d.2} holds due to (a) in Theorem~\ref{m.3}. For (2) of Definition~\ref{d.2}, observe that 
	$$
	\mu(V_n^j) = \mu\left(\bigcup_{A\in\cC_j(V_n)} A\right) \le \mu(V_n) + \mu\left( \bigcup_{A\in\cA^j, A\cap  B_{r_n}(\partial U_n)\ne\emptyset} A\right) \le \mu(V_n) + Cj^{-p'},
	$$
	thanks to (b) in Theorem~\ref{m.3}. A similar argument shows that $\{W_n\}$ is also a good neighborhood system.
	
	Now we can apply Theorem~\ref{m.2} or~\ref{m.gibbsmarkov} on $\{V_n\}$ and $\{W_n\}$ to obtain 
	$$
	\rho(\Lambda, \{W_n\})=\rho(\Lambda, \{V_n\}) = \alpha_1.
	$$
	It then follows from Proposition~\ref{p.approx} that $\rho(\Lambda, \{U_n\}) = \alpha_1$. This concludes the proof of Theorem~\ref{m.3}.
\end{proof}

Similar to Theorem~\ref{t.leftmixing}, when the extremal index $\theta$ is defined as 
$$
\theta = \lim_{n\to\infty}\mu_{U_n}(\tau_{U_n}>K_n)
$$
for some sequence $K_n > \kappa_n^2$, the conditions on the right $\phi$-mixing and $V_n^j$ can be dropped. We thus obtain the following version of Theorem~\ref{t.leftmixing} for open sets $\{U_n\}$:

\begin{theorem}\label{t.leftmixing.open}
	Assume that $T:\bM\to \bM$ is a dynamical system, either invertible or non-invertible, and preserves a measure $\mu$ that is left $\phi$-mixing  with $\phi(k)\le Ck^{p}$ for some  $C>0$ and $p>1$.
	
	Let $\vp:\bM\to \RR\cup\{+\infty\}$ be a  continuous function achieving its maximum on a measure zero set $\Lambda$. Let $\{u_n\}$ be a non-decreasing sequence of real numbers with $u_n\nearrow \sup \vp$, such that the open sets $U_n$ defined by~\eqref{e.Un} satisfy Assumption~\ref{a.1} for a sequence $r_n$ that decreases to $0$ as $n\to\infty$.
	  Let $\kappa_n$ be the smallest positive integer for which $\diam\cA^{\kappa_n}\le r_n$ and assume that:
	\begin{enumerate}[label=(\roman*)]
		\item $\kappa_n\mu(U_n)^\varepsilon\to 0$ for some $\varepsilon \in (0,1)$;
		\item the extremal index $\theta$ defined by~\eqref{e.extremal} exists for some sequence $K_n > \kappa_n^2$.
	\end{enumerate}
	Then the exceedance rate of $\vp$ along $\{u_n\}$ exists and satisfies
	$$
	\zeta(\vp,\{u_n\}) = \rho(\Lambda, \{U_n\}) =  \theta.
	$$
	
\end{theorem}

\section{The conditional escape rate}\label{s.6}
In this section we will prove Theorem~\ref{m.4}.

First we establish the following relation between the hitting times and return times. 

\begin{lemma}
    For any set $U\subset M$ with $\mu(U)>0,$ let $A_k:=\{x\in \bM|\tau_U\geq k\},$ and $B_k:=\{x\in U|\tau_U\geq k\}=A_k\cap U.$ Then we have \begin{equation}
       \mu_U(A_k)\mu(U) =\mu(B_k)=\mu(A_k)-\mu(A_{k+1})
    \end{equation}
\end{lemma}
\begin{proof}
    By definition we have $A_{k+1}\subset A_k.$ Thus, we compute 
    \begin{align*}
        \mu(A_{k+1})&=\mu(\cap_{j=1}^k T^{-j}U^c)\\
        &=\mu(T^{-1}(\cap_{j=0}^{k-1} T^{-j}U^c))\\
        &=\mu(U^c\cap_{j=1}^{k-1}T^{-j}U^c)\\
        &=\mu(U^c\cap A_k)\\
        &=\mu(A_k)-\mu(U\cap A_k)\\
        &=\mu(A_k)-\mu(B_k),
    \end{align*}
    where the third equality follows from the invariance of $\mu.$
\end{proof}

Next, we need the following arithmetic lemma on the exponential decay rate for a sequence of real numbers $\{a_n\}$ and its difference sequence $\{b_n = a_n - a_{n+1}\}$.
\begin{lemma}
    Suppose that $\{a_n\}$ is a decreasing sequence of positive real numbers with $a_n\searrow 0$.  Let $b_n=a_n-a_{n+1}.$ Suppose, also, that $b_n$ is monotonically decreasing. Then the following statements are equivalent:
    \begin{enumerate}
    \item $\lim_{n\ra\infty}-\frac{\log a_n}{n}=\vartheta$ for some $\vartheta>0$;
    \item $\lim_{n\ra\infty}-\frac{\log b_n}{n}=\vartheta$ for some $\vartheta>0$.
    \end{enumerate}
\end{lemma}
\begin{remark}
    Note that there are counter-examples for which the statement of the lemma fails without the monotonicity assumption on the sequence  $\{b_n\}$.
\end{remark}
\begin{proof}
First note that $a_n = \sum_{k\ge n} b_k$; therefore (2) $\implies $ (1) is obvious. 
It thus remains to show that (1) $\implies$ (2).
Let $\gamma>1$ be fixed.
Since by assumption the limit $\lim_{n\to\infty}\frac1n|\log a_n|=\vartheta$ exists and $\vartheta>0$,
 there is an $N$ so that 
$$
\left|\frac{\log a_n}n+\vartheta\right|<\varepsilon \;\;\; \forall n\ge N,
$$
for some positive $\varepsilon<(\gamma-1)\vartheta/4$.
Hence
$$
\left|\frac{\log a_n}n-\frac{\log a_{\gamma n}}{\gamma n}\right|<2\varepsilon \;\;\; \forall n\ge N
$$
which implies 
$$
a_{\gamma n}<\left(a_ne^{2\varepsilon n}\right)^\gamma<\frac12a_n
$$
for all $n$ large enough (assuming $\varepsilon<\frac\vartheta2$).
Since 
$$
a_n-a_{\gamma n}=\sum_{j=n}^{\gamma n-1}b_j
$$
we get by monotonicity of the sequence $b_j$
$$
b_{\gamma n}(\gamma-1)n\le a_n-a_{\gamma n}\le b_n(\gamma-1)n
$$
and consequently $a_n\le 2b_n(\gamma-1)n$.
Hence
$$
\frac{\log b_{\gamma n}}n +\frac{\log(\gamma-1)n}n
\le\frac{\log a_n}n
\le \frac{\log b_n}n+\frac{\log 2(\gamma-1)n}n
$$
which in the limit $n\to \infty$ yields
$$
\limsup_{n\to\infty}\frac{\log b_n}n\le -\frac\vartheta\gamma
$$ 
and 
$$
-\vartheta\le \liminf_{n\to\infty}\frac{\log b_n}n.
$$
Since this applies to every $\gamma>1$ we obtain that 
$\lim_{n\to\infty}\frac1n\log b_n=-\vartheta$.
\end{proof}

\begin{proof}[Proof of Theorem~\ref{m.4}]

The theorem now follows with $a_k=\mu(A_k) = \PP(\tau_U\ge k)$, and $b_k=\mu(U)\mu(B_k) = \mu(U)\PP_U(\tau_U\ge k).$ 

\end{proof}

\section{Escape rate under inducing}\label{s.7}

In this section, we will state a general theorem for the local escape rate under inducing. For this purpose, we consider  a measure preserving dynamical system $(\tom, \tT,\tmu)$  with $\tmu$ being a probability measure. Given a measurable function $R: \tom \to \ZZ^+$ consider the space $\Omega = \tom\times \ZZ^+ /\sim$ with the equivalence relation $\sim$ given by
$$
(x,R(x))\sim (\tT(x), 0).
$$
Define {\em the (discrete-time) suspension map over $\tom$ with roof function $R$} as the measurable map $T$ on the space $\Omega$ acting by
$$\
T(x,j)=\left\{\begin{array}{ll}(x,j+1) &\mbox{if } j<R(x)-1,\\
(\hat{T}x,0)&\mbox{if } j=R(x)-1.\end{array}\right.
$$
We will call $\Omega$ a {\em tower over $\tom$} and refer to the set $\Omega_k:=\{(x,k): x\in\tom, k<R(x)\}$ as the {\em $k$th floor} where
$\tom$ can be naturally identified with the $0$th floor called {\em the base of the tower}.

For $0\le k < i$, set $\Omega_{k,i} = \{(x,k): R(x) = i\}$. The map
$$
\Pi: (x,k)\mapsto x
$$
is naturally viewed as a projection from the tower $\Omega$ to the base $\tom$ and for any given set $U\subset\Omega$ we will write
$$
\tU = \Pi(U).
$$

The measure $\tmu$ can be lifted to a measure $\hat\mu$ on $\Omega$ by
$$
\hat\mu(A) = \sum_{i=1}^{\infty}\sum_{k=0}^{i-1} \tmu(\Pi(A\cap \Omega_{k,i})).
$$
It is easy to verify that $\hat\mu$ is $T$-invariant and if $\tmu(R) = \int R\, d\tmu<\infty$ then $\hat\mu$ is a finite measure. In this case, the measure
$$
\mu = \frac{\hat\mu}{\tmu(R)}
$$
is a $T$-invariant probability measure on $\Omega$.





We write $\tilde U = \Pi(U)\subset \tom$, $\tilde\Lambda = \cap_n\tilde U_n$  and define $\tilde\rho(\tilde\Lambda, \{\tilde U_n\})$ to be the localized escape rate at $\tilde\Lambda$ for the system $(\tom, \tT,\tmu)$. The following theorem relates the escaped rate of the base system with that of the suspension. A similar result is obtained for continuous suspensions under the assumption that $R$ is bounded, see~\cite{DK}.

\begin{theorem}\label{t.suspension}
	Let $(\Omega, T, \mu)$ be a discrete-time suspension over an ergodic measure preserving system $(\tom, \tT, \tmu)$ with a roof function $R$ satisfying the following assumptions:
	\begin{enumerate}
		\item {\em $R$ has exponential tail:} there exists $C, c>0$ such that $\tmu(R>n) \le Ce^{-cn}$;
		\item {\em exponential large deviation estimate:} for every $\varepsilon>0$ small, there exists $C_\varepsilon, c_\varepsilon>0$ such that the set
		$$
		B_{\varepsilon, k} = \left\{y\in \tom:  \left|\frac{1}{n}\sum_{j=0}^{n-1}R(\tilde{T}^jy_0) - \frac{1}{\mu(\Omega_0)} \right|>\varepsilon \mbox{ for some } n\ge k\right\},
		$$	
		satisfies $\tilde\mu(B_{\varepsilon,k}) \le C_\varepsilon e^{-c_\varepsilon k}$.
	\end{enumerate}
	Then for every nested sequence 	$\{U_n\}$, we have 
	$$
	\rho(\Lambda, \{U_n\}) = \tilde \rho(\tilde\Lambda, \{\tilde U_n\}).
	$$ 
\end{theorem}

\begin{proof}
	The result of this theorem is in fact hidden in the proof of Theorem 4 of~\cite{HY} and Theorem 3.2 (1) in~\cite{BDT}. We include the proof here for completeness.
	
	For every $y= (x,m)\in\Omega$, we take $y_0 = x\in \tom$.	Then we have
	\begin{equation}\label{e.0}
	\tau_{U_n}(y) =-m+\sum_{j=0}^{\tilde\tau_{\tilde U_n}(y_0)-1} R(\tilde{T}^j(y_0)),
	\end{equation}
	where $\tilde\tau$ is the return times defined for the system $(\tom, \tT,\tmu)$. 
	By the Birkhoff ergodic theorem on $(\tom,\tilde{T},\tilde\mu)$, we see that
	$$
	\frac{1}{n}\sum_{j=0}^{n-1}R(\tilde{T}^jy_0)\to \int_{\tom} R(y)\,d\tilde\mu(y) = \frac{1}{\mu(\Omega_0)},
	$$
	where we apply the Kac's formula on the last equality and use the fact that $\mu$ is the lift of $\tilde\mu$.
	
	On the other hand, since the return time function $R$ has exponential tail, we get, for each $\varepsilon>0$ and $t$ large enough,
	$$
	\mu((x,m):m>\varepsilon t)\lesssim e^{-c\varepsilon t}.
	$$
	To simplify notation, we introduce the set ($n$ is fixed) 
	$$A_t=\left\{ y = (x,m):m<\varepsilon t,\sum_{j=0}^{\tilde\tau_{\tilde U_n}(y_0)-1} R(\tilde {T}^j(y_0))>(1+\varepsilon)t\right\}
	\cap B^c_{\varepsilon,k}.$$ 
	Combine~\eqref{e.0} with the previous  estimates on $B_{\varepsilon,k}$, for  $k=t(1+\varepsilon)$ we get
	\begin{equation}\label{e.1}
	\left|\mu(\tau_{\tilde U_n}>t) - \mu(A_t)\right|
	\lesssim e^{-c\varepsilon t}+e^{-c_\varepsilon (1+\varepsilon) t}.
	\end{equation}
	
	Note that  $A_t$ contains the set
	$$
	A_t^- = \left\{y:m<\varepsilon t, \tilde \tau_{\tilde U_n}(y_0)>\frac{(1+\varepsilon)t}{\mu^{-1}(\Omega_0)-\varepsilon}\right\},
	$$
	and is contained in 
	$$
	A_t^+ = \left\{y:m<\varepsilon t, \tilde\tau_{\tilde U_n}(y_0)>\frac{(1+\varepsilon)t}{\mu^{-1}(\Omega_0)+\varepsilon}\right\}.
	$$
	Now we are left to estimate $\mu(A_t^\pm)$. Since $\mu$ is the lift of $\tilde \mu$, we have
	\begin{align}\label{e.25}
	\mu(A^\pm_t) = &\frac{1}{\tilde\mu(R)}\sum_{j=0}^{\infty}\sum_{i=0}^{\min(\varepsilon t, R_j)}\tilde\mu(T^{-i}A^\pm_t\cap \Omega_{0,i})\\\nonumber
	=&\mu(\Omega_0)(1+\mathcal{O}(\varepsilon t))\tilde\mu(\tilde{A}^\pm_t),
	\end{align}
	where
	$$
	\tilde{A}^\pm_t = \left\{y_0\in\Omega_0: \tilde\tau_{\tilde U_n}(y_0)>\frac{(1+\varepsilon)t}{\mu^{-1}(\Omega_0)\pm\varepsilon}\right\}.
	$$
	
	Let $\alpha = \tilde\rho(\tilde\Lambda, \{\tilde U_n\})$. Then we have (recall that $\tilde\mu(\tilde U_n) \mu(\Omega_0)=\mu(U_n)$)
	$$
	\lim_{n\to\infty}\lim_{t\to\infty}\frac{1}{t\mu(U_n)}|\log\tilde\mu(\tilde{A}^\pm_t)| =\alpha \frac{(1+\varepsilon)}{1\pm\varepsilon\mu(\Omega_0)}.
	$$
	By~\eqref{e.25}, we get that 
	$$
	\lim_{n\to \infty}\lim_{t\to\infty}\frac{1}{t\mu(U_n)}|\log\mu(\tilde{A}^\pm_t)| = \alpha \frac{(1+\varepsilon)}{1\pm\varepsilon\mu(\Omega_0)}.
	$$
	
	For each $\varepsilon>0$ we can take $n_0$ large enough, such that for $n>n_0:$
	$$
	\alpha\frac{1+\varepsilon}{1\pm\varepsilon \mu(\Omega_0)}\mu(U_n) < \min\{c\varepsilon ,c_\varepsilon(1+\varepsilon)\}. 
	$$
	It then follows that the right-hand-side of \eqref{e.1} is of order $o(\mu(A^\pm_t)).$ We thus obtain
	$$
	\rho(\Lambda,\{U_n\})=\lim_{n\to\infty}\lim_{t\to\infty}\frac{1}{t\mu(U_n)}|\log\mu(\tau_{U_n}>t)| \in \left(\alpha\frac{(1+\varepsilon)}{1+\varepsilon\mu(\Omega_0)},\alpha\frac{(1+\varepsilon)}{1-\varepsilon\mu(\Omega_0)}\right)
	$$
	for every $\varepsilon>0$. This shows that $\rho(\Lambda, \{U_n\})=\alpha = \tilde\rho(\tilde\Lambda, \{\tilde U_n\})$.
\end{proof}

\begin{proof}[Proof of Theorem~\ref{m.5}]
	Young towers can be seen as discrete-time suspension over Gibbs-Markov maps. Moreover, the exponential tail  of $\mu(R>n)$ implies the exponential large deviation estimate (see for example~\cite{BDT} Appendix B). Therefore Theorem~\ref{m.5} immediately follows from Theorem~\ref{m.gibbsmarkov}, Theorem~\ref{m.3} and Theorem~\ref{t.suspension}.
\end{proof}

\section{Examples}\label{s.8}

\subsection{Periodic and non-periodic points dichotomy}
First we consider the case where $\Lambda = \{x\}$ is a singleton, and $U_n = B_{\delta_n}(x)$ is a sequence of balls shrinking to $x$. Alternatively one could take $\varphi(y) = g(d(y,x))$ for some function $g(x): \RR\to \RR\cup\{+\infty\}$ achieving its maximum at $0$ (for example, $g(y) = -\log y$) and let $u_n\nearrow\infty$ be a sequence of threshold tending to infinity. Then $U_n = \{y:\varphi(y) > u_n\}$ is a sequence of balls with diameter shrinking to zero.

This situation has been dealt with in~\cite{BDT} for certain interval maps, and in~\cite{HY} for maps that are polynomially $\phi$-mixing. A dichotomy is obtained: when $x$ is non-periodic the local escape rate is $1$; when $x$ is periodic then $\rho(x) = 1-\theta$ where 
\begin{equation}\label{e.theta}
\theta = \theta(x) = \lim_{n\to\infty} \frac{\mu(U_n\cap T^{-m}U_n)}{\mu(U_n)},
\end{equation}
where $m$ is the periodic of $x$. When $\mu$ is an equilibrium state for some potential function $h(x)$ with zero pressure, one has $\theta = e^{S_mh(x)}$ where $S_m$ is the Birkhoff sum. See~\cite{BDT}.

Note that if $x$ is non-periodic then one naturally deduces that $\pi(U_n)\nearrow\infty$ (see for example~\cite[Lemma 1]{HY}). When $x$ is periodic, in~\cite[Section 8.3]{HV19} it is shown that $\hat\alpha_\ell = \theta^{l-1}$ is a geometric distribution. In particular one has $\sum_\ell \ell\hat\alpha_\ell<\infty$ and $\alpha_1 = 1-\theta.$ This leads to the following theorem:

\begin{theorem}	
	Assume that 
	\begin{enumerate}
		\item either $\mu$ is right $\phi$-mixing with $\phi(k)\le Ck^{-p}$, $p>1$;
		\item or $(T,\mu,\cA)$ is a Gibbs-Markov system.
	\end{enumerate}
	Assume that $0<r_n<\delta_n$ satisfies 
	$$
	\mu(B_{\delta_n + r_n}(x)\setminus B_{\delta_n - r_n}(x)) = o(1) \mu(B_{\delta_n}(x)).
	$$
	Write $\kappa_n$ for the smallest positive integer with $\diam\cA^{\kappa_n}\le r_n$.  We assume that:
	\begin{enumerate}[label=(\alph*)]
		\item $\kappa_n\mu(U_n)^\varepsilon\to 0$ for some $\varepsilon \in (0,1)$;
		\item $U_n$ has small boundary: there exists $C>0$ and $p'>1$, such that \\$\mu\left(\bigcup_{A\in\cA^j, A\cap  B_{r_n}(\partial U_n) \ne\emptyset}A\right)\le C j^{-p'}$ for all $n$ and $j\le\kappa_n$.
		\item when $x$ is periodic with period $m$, $\theta$ defined by~\eqref{e.theta} exists.
	\end{enumerate}
	Then we have
	$$
	\rho(\{x\},\{B_{\delta_n}(x)\}) = \alpha_1 = \begin{cases}
	1&\mbox{if $x$ is non-periodic}\\
	1-\theta&\mbox{if $x$ is periodic}\end{cases}.
	$$
\end{theorem}
This theorem improves~\cite[Theorem 2]{HY} by dropping the assumption $\theta < 1/2$. Also note that such results can be generalized to interval maps which can be modeled by Young's towers using Theorem~\ref{m.5}.

\subsection{Cantor sets for interval expanding maps}
For simplicity, below we will only consider the Cantor ternary set. However the argument below can be adapted to a large family of dynamically-defined Cantor set discussed in~\cite{FFRS} with only minor modification.

Consider the uniformly expanding map $T(x) = 3x \mod 1$ defined on the unit interval $[0,1]$. We take $\Lambda$ to be the ternary Cantor set on $[0,1]$, and define recursively:
$
U_0 = [0,1]; 
$
$U_{n+1}$ is obtained by removing the middle third of each connected component of  $U_{n}$. 
Then we have $\cap_n U_n = \Lambda$.

\begin{theorem}
	For the uniformly expanding map $T(x) = 3x\mod 1$ on $[0,1]$, the Cantor ternary set $\Lambda$ and the nested sets $\{U_n\}$, we have 
	$$
	\rho(\Lambda, \{U_n\}) = \frac13.
	$$
\end{theorem}

\begin{proof}
Let $\cA = \{[0,1/3), [1/3, 2/3), [2/3, 1]\}$ be a Markov partition of $T$, with respect to which the Lebesgue measure $\mu$ is exponentially $\psi$-mixing.  Below we will verify the assumptions of Proposition~\ref{p.4.9}.

It is easy to see that $U_n \in \cA^n$, i.e., $\kappa_n = n$. On the other hand, $\mu(U_n) = 2^n/3^n$ which shows that item (1) of Definition~\ref{d.2} is satisfied for any $\varepsilon \in (0,1)$. For item (2), note that $U_n^j = U_j$
which implies that 
$$
\mu(U_n^j) \le \mu(U_n) + \mu(U_j) = \mu(U_n) + \left(\frac{2}{3}\right)^j.
$$
We conclude that $\{U_n\}$ is a good neighborhood system. 

The extremal index can be found as follows. Let us write
$U_n=\bigcup_{|\alpha|=n} J_\alpha$ where the disjoint union is over all $n$-words 
$\alpha=\alpha_1\alpha_2\dots\alpha_n\in\{0,2\}^n$ and 
$$
J_\alpha =\sum_{j=1}^n\frac{\alpha_j}{3^j}+3^{-n}I,
$$
where $I=[0,1]$ is the unit interval. The length $|J_\alpha|$ is equal to $3^{-n}$. 
For $j<n$
$$
T^{-j}J_\alpha=\bigcup_{\beta\in\{0,1,2\}^j} J_{\beta\alpha}
$$
(disjoint union), thus 
$$
U_n\cap T^{-j}U_n
=\bigcup_{\alpha\in\{0,2\}^n}\bigcup_{\beta\in\{0,2\}^j}J_{\beta\alpha}
$$
and therefore $U_n\cap T^{-j}U_n=U_{n+j}$. Consequently
$$
\{\tau_{U_n}\le K\}\cap U_n
=U_n\cap\bigcup_{j=1}^KT^{-j}U_n=U_{n+1}.
$$
Since $\mu(U_{n+j})=\!\left(\frac23\right)^j\mu(U_n)$
this implies that $\hat\alpha_2(K,U_n)=\frac{\mu(U_{n+1})}{\mu(U_n)}=\frac23$
and therefore $\alpha_1=\frac13$.

This result was recently shown in~\cite[Theorem 3.3]{FFRS} in more generality.
By Proposition~\ref{p.4.9} we conclude that $\rho(\Lambda, \{U_n\})$ = 1/3. 
\end{proof}

\subsection{Submanifolds of Anosov maps}\label{s.8.2}
In this section we consider the case where $\Lambda$ is a submanifold for some Anosov map $T$. More importantly, we will show how our results can be applied to those cases where the extremal index $\theta$ is defined using time cut-off $K_n$ that depends on $U_n$ (see~\eqref{e.extremal}).

Let $T = \begin{pmatrix}
2 & 1\\
1 & 1
\end{pmatrix}$ be an Anosov system on $\TT^2$ and $\mu$ be the Lebesgue measure. It is well known that $\mu$ is exponentially $\psi$-mixing with respect to its Markov partition $\cA$. Also denote by $\lambda>1$ the eigenvalue of $T$. Following~\cite{CHN} we take $\Lambda$ to be a line segment with finite length $l(\Lambda)$. We will lift $\Lambda$ to $\hat\Lambda\subset \RR^2$ and parametrize $\hat\Lambda$ by $p_1 + t v$ for some $p_1\in\RR^2$ and $t\in[0, l(\Lambda)]$. Write $p_2$ for the other end point of $\hat\Lambda$, that is, $p_2 = p_1 + l(\Lambda) v$.

Consider the function $\varphi_\Lambda(y) = -\log d(x, \Lambda)$ which achieves its maximum ($+\infty$) on $\Lambda$. Write $v^{*}, * = s, u$ for the unit vector along the stable and unstable direction respectively. Then we have:

\begin{theorem}
	For the sequence $\{u_n = \log n\}$, 
	\begin{enumerate}
		\item if $\Lambda$ is not aligned with the stable direction $v^s$ or the unstable direction $v^u$ then $\zeta(\varphi_\Lambda, \{u_n\}) = 1$;
		\item if $\Lambda$ is aligned with the unstable direction but $\{p_1 + tv^u, t\in\RR\}$ has no periodic point, then $\zeta(\varphi_\Lambda, \{u_n\}) = 1$;
		\item if $\Lambda$ is aligned with the stable direction but $\{p_1 + tv^s, t\in\RR\}$ has no periodic point, then $\zeta(\varphi_\Lambda, \{u_n\}) = 1$;
		\item $\Lambda$ is aligned with $v^{*}, *=s,u$ and $L$ contains a periodic point with prime period $q$, then $\zeta(\varphi_\Lambda, \{u_n\}) = 1 - \lambda^{-q}$;
		\item $\Lambda$ is aligned with the unstable direction $v^u$, $\Lambda$ has no periodic points but $\{p_1+tv^u, t\in\RR\}$ contains a periodic point of prime period $q$; if $\Lambda \cap T^{-q}\Lambda = \emptyset $ then $\zeta(\varphi_\Lambda, \{u_n\}) = 1$; if $\Lambda\cap T^{-q} \Lambda\ne\emptyset$ then $\zeta(\varphi_\Lambda, \{u_n\}) = (1 - \lambda^{-q})\frac{|p_2|}{l(\Lambda)}$;
		\item $\Lambda$ is aligned with the stable direction $v^u$, $\Lambda$ has no periodic points but $\{p_1+tv^u, t\in\RR\}$ contains a periodic point of prime period $q$; if $\Lambda \cap T^{-q}\Lambda = \emptyset $ then $\zeta(\varphi_\Lambda, \{u_n\}) = 1$; if $\Lambda\cap T^{-q} \Lambda\ne\emptyset$ then $\zeta(\varphi_\Lambda, \{u_n\}) = (1 - \lambda^{-q})\frac{|p_2|}{l(\Lambda)}$;
  	\end{enumerate}
\end{theorem}

\begin{proof}
	We will only prove case (1), in which we will need the result of \cite[Theorem 2.1 (1)]{CHN}. The other cases use similar arguments and correspond to case (2) to (6) of~\cite[Theorem 2.1]{CHN}.

	Below we verify the assumptions of Theorem~\ref{t.leftmixing.open}.
	
	Put $\delta_n = e^{-u_n}$. Then we see that $U_n = \{y: \varphi_\Lambda(y) > u_n\} = B_{\delta_n}(\Lambda)$. 
	Since $\mu$ is the Lebesgue measure, it is straight forward to verify that Assumption 1 is satisfied with 
	$r_n = \delta_n^2 = e^{-2 u_n}$. 
	
	By the hyperbolicity of $T$, there exists $C>0$ such that $\diam \cA^n < C\lambda^{-n}$. This invites us to take
	$$
	\kappa_n = \floor{\frac{\ln C + 2u_n}{\ln \lambda}}+1 = \cO(\log n)
	$$
	which guarantees that $\diam\cA^{\kappa_n} < r_n$.
	On the other hand, $\mu(U_n)\lesssim e^{-u_n}l(\Lambda) = \cO(1/n)$, so item (i) of Theorem~\ref{t.leftmixing.open} is satisfied for any $\varepsilon \in (0,1)$.
	
	We are left with the extremal index $\theta$ defined by~\eqref{e.extremal}. For this purpose we choose $K_n = (\log n)^5 \gg \kappa_n^2$. Now we estimate:
	\begin{align*}
		\mu_{U_n}(\tau_{U_n} \le K_n)\le& \frac{1}{\mu(U_n)} \sum_{j=1}^{(\log n)^5} \mu(U_n\cap T^{-j} U_n )\\
		\lesssim&\, n\sum_{j=1}^{(\log n)^5} \mu(U_n\cap T^{-j} U_n )\\
		 =&\, o(1)
	\end{align*}
	where the last inequality follows from~\cite[Section 3.3, page 16]{CHN}. This shows that 
	$$
	\theta = \lim_n\mu_{U_n}(\tau_{U_n} >  K_n) = 1- \lim_n\mu_{U_n}(\tau_{U_n} \le K_n) = 1,
	$$
	finishing the proof of (ii) of Theorem~\ref{t.leftmixing.open}. We conclude that 
	$$
	\zeta(\varphi, \{\log n\}) = \theta = 1.
	$$

\end{proof}


\end{document}